\documentclass[12pt]{amsart}
\usepackage{amssymb}
\usepackage{graphicx}
\usepackage{color}
\usepackage{latexsym}
\usepackage{amssymb,amsmath,amstext}
\usepackage{epsfig}
\usepackage{graphics}
\usepackage{amsthm}
\usepackage{verbatim}
\usepackage{placeins}
\usepackage[colorlinks=true, urlcolor=blue, linkcolor=blue, citecolor=blue]{hyperref}
\usepackage[margin=1in]{geometry}
\numberwithin{equation}{section}
\theoremstyle{plain}%
\newtheorem{theorem}{Theorem}
\numberwithin{theorem}{section}
\newtheorem{proposition}[theorem]{Proposition}
\newtheorem{example}[theorem]{Example}
\newtheorem{lemma}[theorem]{Lemma}
\newtheorem{corollary}[theorem]{Corollary}
\newtheorem{definition}[theorem]{Definition}

\allowdisplaybreaks

\begin{document}

\title{\bf Tropical hyperelliptic curves in the plane}

\author{Ralph Morrison
}
\address[Ralph Morrison]{Department of Mathematics and Statistics, Williams College, Williamstown, MA 01267}
\email{10rem@williams.edu}

\date{}

\maketitle
\begin{abstract}
Abstractly, tropical hyperelliptic curves are metric graphs that admit a two-to-one harmonic morphism to a tree.  They also appear as embedded tropical curves in the plane arising from triangulations of polygons with all interior lattice points collinear.  We prove that hyperelliptic graphs can only arise from such polygons.  Along the way we will prove certain graphs do not embed tropically in the plane due to entirely combinatorial obstructions, regardless of whether their metric is actually hyperelliptic.
\end{abstract}
\section{Introduction}
\label{intro}
Tropical curves can be defined in either an abstract or an embedded way.  Abstractly, they are connected, weighted metric graphs.  Often these are stratified by topological genus $g\geq 2$, and are parametrized by the stacky fan $\mathbb{M}_g$, the \emph{moduli space of tropical curves of genus $g$}.  A point in this space is a metric graph with integer weights on its vertices, such that the sum of the weights plus the first Betti number of the graph sum to $g$.  Like the classical moduli space of curves $\mathcal{M}_g$, this space has dimension $3g-3$, and a strong connection between these spaces is established in \cite{ACP}.  In particular, one associates to a curve the dual graph of the special fiber, weighting a vertex according to the genus of the corresponding irreducible component, and assigning edge lengths according to valuations associated to intersection points of components.  This graph sometimes referred to as the \emph{tropicalization} of the curve $C$, written $\textrm{trop}(C)$. This relationship to algebraic curves can be viewed as a primary motivation for studying tropical curves:  proving results for tropical curves leads to results for algebraic curves.  For instance, determining what is known as a metric graph's \emph{divisorial gonality} provides a lower bound on the gonality of a related algebraic curve \cite{baker}.

The space $\mathbb{M}_g$ can be constructed as follows.  For each (non-metric) trivalent graph $G$ with first Betti number $g$, we construct the space of all choices of lengths on the $3g-3$ edges of $G$ as $(\mathbb{R}_{\geq 0})^{3g-3}/\textrm{Aut}(G)$, where $\textrm{Aut}(G)$ is the automorphism group of $G$.  These spaces are then glued together according to a poset relating how different trivalent graphs can be transformed to the same graph by shrinking edge lengths to $0$.  Whenever a loop is collapsed to a vertex, that vertex is given an integer weight to record the decrease in Betti number.  For the majority of this paper, we will be concerned with metric graphs in the top strata of this space, i.e. those that are trivalent with no vertex weights.  We refer the reader to \cite{BMV,Chan,Chan_lectures} for more background on $\mathbb{M}_g$.

The perspective that views tropical curves as embedded objects considers a tropical curve as a one-dimensional rational weighted balanced polyhedral complex in $\mathbb{R}^n$, arising as the non-linear locus of a collection of polynomials over the min-plus algebra, as presented in \cite{MS}.  Such a tropical curve inherits a metric from the $\mathbb{Z}^n$ lattice, and contains a distinguished metric graph called the \emph{skeleton}, minimal among the subgraphs admitting a deformation retract of the whole tropical curve.

\begin{figure}
\centering
\includegraphics{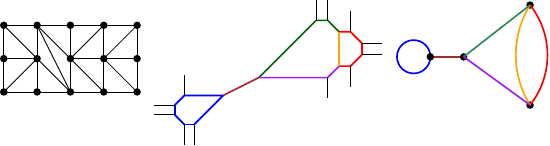}
\caption{A triangulation, a dual tropical curve, and its skeleton}
\label{figure:first_example}
 \end{figure}
 
 Embedded tropical curves arise from algebraic ones through another tropicalization process.  Let $K$ be an algebraically closed field, complete with respect to a nontrivial non-Archimedean valuation.   Consider an algebraic curve $C$ in the $n$-dimensional torus $(K^*)^n$.  Apply coordinate-wise valuation to the points of $C$, and take the Euclidean closure in $\mathbb{R}^n$ of the image of this map.  The resulting subset $\Gamma$ is a one-dimensional rational polyhedral complex, and can be given weights in a natural way to make it balanced. See \cite[\S 4]{MS} for more details on this tropicalization  process, and for the connection to min-plus polynomials.  There is a strong relationship between this process and the association of a metric graph to a point in $\mathcal{M}_g$:  if the tropical curve is trivalent and unweighted, then it ``faithfully'' represents the associated metric graph \cite[Corollary 5.28]{BPR}.  Finding embeddings of algebraic curves that give faithful tropicalizations is an important task.  One such approach uses a tool called \emph{tropical modification}, which is used in \cite{HMRT} to find such embeddings for all nonhyperelliptic curves of genus $3$.

The embedded tropical curves we focus on in this paper appear in the Euclidean plane.  A tropical plane curve $C$ is a one-dimensional weighted balanced polyhedral complex in $\mathbb{R}^2$, dual to a regular subdivision of the Newton polygon $P$ of the defining polynomial $f(x,y)$ of the curve.  Under this duality, polygons correspond to vertices, interior edges correspond to edges, boundary edges correspond to rays, and lattice points correspond to two-dimensional regions.  If this subdivision is a unimodular triangulation, we say that $C$ is \emph{smooth}.  (It is worth remarking that if a tropicalization of an algebraic curve is a smooth tropical plane curve, then that tropical curve is trivalent with no weights, and so it is a faithful tropicalization in the language of \cite{BPR}. This justifies the borrowing of the word ``skeleton'' from the terminology of Berkovich theory.)  We refer to the number of interior lattice points of $P$ as the \emph{genus} of $P$.  If $P$ has genus $g$, and $C$ is smooth, then then the skeleton of $C$ is a trivalent, unweighted metric graph of genus $g$.  In this paper we restrict our attention to such graphs. An example of such a tropical plane curve, together with the triangulation and the skeleton, appears in Figure \ref{figure:first_example}.

Let $g\geq 2$, and let $P$ be a convex lattice polygon with genus $g$.  We denote by $\mathbb{M}_P$ the subset of $\mathbb{M}_g$ consisting of all metric graphs that arise as the skeleton of a smooth tropical plane curve, up to closure.  We then define
the \emph{moduli space of tropical plane curves of genus $g$}, denoted $\mathbb{M}_g^{\textrm{planar}}$, as the union of all such $\mathbb{M}_P$.  Thus $\mathbb{M}_g^{\textrm{planar}}$ is the subset of $\mathbb{M}_g$ consisting of all metric graphs of genus $g$ that arise as the skeleton of a smooth tropical plane curve, up to closure  \cite{BJMS}.  For $g\geq 3$, $\mathbb{M}_g^{\textrm{planar}}$ is strictly contained in $\mathbb{M}_g$:  there are \emph{combinatorial} obstructions, meaning that certain types of graphs never arise as the skeleton of a smooth tropical plane curve, regardless of the metric; and there are \emph{metric} obstructions, meaning that some but not all choices of edge lengths on a graph $G$ arise from a smooth tropical plane curve.  For instance, \cite[Theorem 1.1]{BJMS} shows that for genus $g\geq 4$ and $g\neq 7$, $\textrm{dim}(\mathbb{M}_g^{\textrm{planar}})=2g+1$, which is strictly smaller than  $\textrm{dim}(\mathbb{M}_g)=3g-3$ for $g\geq 5$, indicating that not all metric graphs appear in $\mathbb{M}_g^{\textrm{planar}}$.

A complete characterization of which metric graphs appear in $\mathbb{M}_g^{\textrm{planar}}$ is in general an open problem.  It has been answered for $g\leq5$ in \cite{BJMS} by enumerating all regular unimodular triangulations of maximal polygons with genus at most $5$, and computing the cone of metrics arising from each such triangulation.  Even the combinatorial question of which types of graphs arise is  difficult:  in genus $5$, there are several non-achievable planar graphs that are not ruled out by any known nice criteria, such as those we will present in Lemmas \ref{prop:sprawling} and \ref{proposition:crowded}.

The moduli space $\mathbb{M}_g^{\textrm{planar}}$ is related to a number of other moduli spaces from algebraic geometry.  In \cite{CV}, the authors consider the moduli space of non-degenerate algebraic curves, which are those curves that admit a sufficiently nice embedding on a toric surface.  Their moduli space $\mathcal{M}_g^{\textrm{nd}}$ has the same dimension as $\mathbb{M}_g^{\textrm{planar}}$, and any metric graph in $\mathbb{M}_g^{\textrm{planar}}$ appears as the tropicalization of some curve in $\mathcal{M}_g^{\textrm{nd}}$, so $\mathbb{M}_g^{\textrm{planar}}\subseteq\textrm{trop}\left(\mathcal{M}_g^{\textrm{nd}}\right)$. However, it turns out that $\textrm{trop}\left(\mathcal{M}_g^{\textrm{nd}}\right)$ is not equal to $\mathbb{M}_g^{\textrm{planar}}$ for $g\geq3$; see \cite[\S 3]{BJMS} and Example \ref{example:nd}.  More generally, the notion of using tropical methods to study curves on toric surfaces is explored in \cite{NS,dhruv}.  A promising topic for future research would be to make more precise the relationship between  $\mathbb{M}_g^{\textrm{planar}}$ and moduli spaces of tropical stable maps to toric surfaces.

In this paper, we focus on \emph{hyperelliptic graphs}, the tropical analog of hyperelliptic curves \cite{Chan2}.  The usual definition relies on a theory of divisors on graphs, developed in \cite{BN} for combinatorial graphs and extended to metric graphs in \cite{GK,MZ}.  The \emph{divisorial gonality} of a graph is the minimum degree of a rank $1$ divisor on that graph. A metric graph is called \emph{hyperelliptic} if it has divisorial gonality $2$; that is, if it has a divisor of degree $2$ and rank $1$.  A nearly equivalent definition, which we will present more thoroughly in Section \ref{section:background}, characterizes hyperelliptic graphs as those metric graphs that admit a degree $2$ map to a tree. The locus of hyperelliptic graphs inside $\mathbb{M}_g$ is denoted $\mathbb{M}_{g,{\textrm{hyp}}}$.  We remark that this set is \emph{not} equivalent to the tropicalization of the classical moduli space of hyperelliptic curves; as shown in \cite{ABBR}, there exist metric graphs with degree $2$ rank $1$ divisors that do not arise as the tropicalization of any hyperelliptic curve. 

There is also a notion of \emph{hyperelliptic polygons}, as defined in \cite{Ca}.  Given a lattice polygon $P$ with $g\geq 2$ interior lattice points, we can consider $P_{\textrm{int}}$, the convex hull of all lattice points of $P$ not on the boundary of $P$:
$$P_{\textrm{int}}=\textrm{conv}((P\cap\mathbb{Z}^2)\setminus \partial P).$$
  The polygon $P$ is called \emph{hyperelliptic} if $P_\textrm{int}$ is a line segment, and \emph{nonhyperelliptic} if $P_{\textrm{int}}$ is a two-dimensional polygon.  Taking the union of $\mathbb{M}_P$ over all hyperelliptic polygons $P$ of genus $g$, we obtain the \emph{moduli space of hyperelliptic tropical plane curves of genus $g$}, denoted $ \mathbb{M}^{\textrm{planar}}_{g,{\textrm{hyp}}}$.  In other words, this is the locus inside of $\mathbb{M}_g^{\textrm{planar}}$ of all metric graphs arising from hyperelliptic polygons with $g$ interior lattice points.  For example, the metric graph illustrated in Figure \ref{figure:first_example} is one point in $ \mathbb{M}^{\textrm{planar}}_{3,{\textrm{hyp}}}$, since it is the skeleton of a smooth tropical plane curve with a hyperelliptic Newton polygon.

It is reasonable to ask about the relationship between $ \mathbb{M}^{\textrm{planar}}_{g,{\textrm{hyp}}}$ and  $\mathbb{M}_{g,{\textrm{hyp}}}$. The easier direction is that the first is contained in the second:  assuming $P_{\textrm{int}}$ is a horizontal line segment, a degree $2$ map from a tropical curve (dual to a subdivision of $P$) to a subdivided line segment is given by vertical projection and bridge-dilation.  This is the content of Lemma \ref{lemma:chains}.   Our main result is the following theorem, which shows that the relationship is as nice as can be hoped for.

\begin{theorem}
\label{thm:onlyhyperelliptic}
If a smooth tropical plane curve with Newton polygon $P$ has a hyperelliptic skeleton, then $P$ is a hyperelliptic polygon.
\end{theorem}
In the language of moduli spaces, if $\mathbb{M}_P$ contains a hyperelliptic graph that did not come from taking closures, we have that $P$ is hyperelliptic.  This means that, \textbf{before taking closures},
$$\mathbb{M}^{\textrm{planar}}_{g,{\textrm{hyp}}}=\mathbb{M}_{g,\textrm{hyp}}\cap \mathbb{M}_g^{\textrm{planar}}. $$
This generalizes  \cite[Theorem 4.3]{BLMPR}, which proved the result for $g=3$. 

We remark that the ``before taking closures'' assumption is absolutely vital:  if we allow edge lengths to go to $0$, which in the limit they can, we can find hyperelliptic graphs appearing in $\mathbb{M}_g^{\textrm{planar}}$ that do not arise from any hyperelliptic polygon, as shown in the following example.

\begin{example}  Figure \ref{figure:poset_portion} illustrates a small portion of the poset of all combinatorial types of graphs of genus $4$, which encodes the data for how the maximal cells of $\mathbb{M}_4$ are glued to one another.  The top two graphs are trivalent, and so correspond to two of the top-dimensional cells of $\mathbb{M}_4$.  Since both can be transformed into the bottom graph by shrinking certain edge lengths, those cells are in part glued according to that relationship; the bottom graph has a marked vertex to record to collapse of a loop.  It was shown in \cite[\S 7]{BJMS} that the upper left graph is the skeleton of the tropically planar graph as long as the inner triangle has all three edge lengths equal.  Thus all such metric graphs appear in $\mathbb{M}_4^{\textrm{planar}}$.  Since we define $\mathbb{M}_4^{\textrm{planar}}$ by taking a closure within $\mathbb{M}_4$ of all achievable skeleta, it follows that the bottom graph also appears in  $\mathbb{M}_4^{\textrm{planar}}$, in fact with any positive edge lengths.  However, these metric graphs also appear in $\mathbb{M}_{4,\textrm{hyp}}$.  To see this, note that as long as the two parallel edges have equal length, the upper right graph is hyperelliptic; and by definition, $\mathbb{M}_{4,\textrm{hyp}}$ is closed under letting edge lengths go to $0$.  Thus $\mathbb{M}_{g,\textrm{hyp}}\cap \mathbb{M}_g^{\textrm{planar}}$ contains all metric graphs with combinatorial type given by the bottom graph.  No such graph can arise from a hyperelliptic polygon:  as we will show in Lemma 2.2, a hyperelliptic polygon can only give rise to skeleta with underlying graphs that are \emph{chains}, and the illustrated graph cannot be obtained from shrinking edge lengths in a chain to $0$.  Thus we have  
$\mathbb{M}^{\textrm{planar}}_{4,{\textrm{hyp}}}\subsetneq\mathbb{M}_{4,\textrm{hyp}}\cap \mathbb{M}_4^{\textrm{planar}}. $

\begin{figure}
\centering
\includegraphics[scale=0.6]{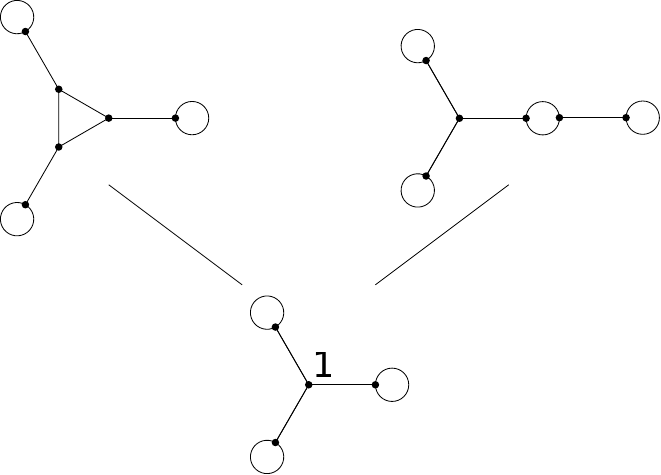}
\caption{Two graphs of genus $4$ with a common minor}
\label{figure:poset_portion}
 \end{figure}

\end{example}

An immediate corollary of Theorem \ref{thm:onlyhyperelliptic} is a lower bound on the divisorial gonality of metric graphs arising from smooth tropical plane curves with nonhyperelliptic Newton polygon.  Let $\textrm{dgon}(G)$ denote the divisorial gonality of a metric graph $G$.

\begin{corollary}  Let $P$ be a nonhyperelliptic polygon, and let $G$ be the skeleton of a smooth tropical plane curve with Newton polygon $P$.  Then $\textrm{dgon}(G)\geq 3$.
\end{corollary}

Combined with \cite{baker}, this lower bound on graph gonality also serves as a lower bound on the algebraic gonality of any algebraic curve tropicalizing to such a tropical curve.  

\begin{proof}  Since $P$ is nonhyperelliptic, we know that $G$ has positive genus, so $G$ is not a tree.  A quick corollary of the tropical Riemann-Roch Theorem \cite{BN,GK,MZ} is that a graph has divisorial gonality $1$ if and only if it is a tree, and so $\textrm{dgon}(G)\neq 1$.  By Theorem \ref{thm:onlyhyperelliptic}, since $P$ is nonhyperelliptic we have that $G$ is not a hyperelliptic graph, so $\textrm{dgon}(G)\neq 2$.  We therefore have $\textrm{dgon}(G)\geq 3$.
\end{proof}
Understanding the relationship between a Newton polygon and the divisorial gonality of dual graphs of its unimodular triangulations is in general a difficult problem \cite[Conjecture 3 + Err.]{CC}.  This corollary provides a lower bound for nonhyperelliptic polygons; other work provides upper bounds using \emph{lattice width} \cite{CC}.

Our strategy for proving Theorem \ref{thm:onlyhyperelliptic} is as follows.   In Section \ref{section:sprawling_and_crowded}, we present two types of graphs, \emph{sprawling} and \emph{crowded}, which are never the skeletons of smooth tropical plane curves.  In Section \ref{section:combinatorial} we  use these criteria to show that if a hyperelliptic graph  is a smooth tropical plane curve's skeleton, then it must be a simple type of graph called a \emph{chain}.   It then suffices to show in Section \ref{section:metric} that if a polygon gives rise to a  chain that is also a hyperelliptic graph, then that polygon must be a hyperelliptic polygon.

\section{Background and definitions}
\label{section:background}

We begin by giving background on the theory of abstract metric graphs, and move from there to embedded tropical curves.

Throughout this paper, all graphs are connected, with loops and multiple edges allowed unless explicitly stated otherwise. We will also focus on graphs with all valencies at most $3$. This means that a triple edge could only occur if a graph consists of two vertices joined by three edges.  Outside of this one graph, all instances of multiple edges will be \emph{bi-edges}, where two vertices are joined by a pair of edges.  If removing an edge of a graph disconnects it, we call that edge a \emph{bridge}.  If a graph has no bridges, we say it is \emph{$2$-edge-connected}.  If we delete all bridges from a graph, the resulting collection of connected graphs are called the \emph{$2$-edge-connected components} of the graph.

For a graph $G$ let $V(G)$ and  $E(G)$ denote its vertices and edges, respectively.  A \emph{metric graph} $(G,\ell)$ is a graph $G$ together with a function $\ell:\,E(G)\rightarrow\mathbb{R}_{>0}$, which assigns positive real lengths to the edges of $G$.  When the metric is clear from context, we may write $(G,\ell)$ simply as $G$.
  
  Given two metric graphs $(G,\ell)$ and $(G',\ell')$ with no loops, we will define a morphism of graphs between them.  We can extend this definition to morphisms of metric graphs with loops allowed by identifying two metric graphs when they define are identical as topological spaces, meaning we may introduce valence $2$ vertices on the interior of any edge. A \emph{morphism} $\phi$ from $(G,\ell)$ to $(G',\ell')$ is a map of sets
  $$\phi:V(G)\cup E(G)\rightarrow V(G')\cup E(G')$$ where $\phi(V(G))\subset V(G')$, and where the following properties hold for all $e\in E(G)$ with endpoints $x$ and $y$:
  \begin{itemize}
\item[(i)] if $\phi(e)\in V(G')$ then $\phi(x)=\phi(e)=\phi(y)$,
\item[(ii)] if $\phi(e)\in E(G')$ then $\phi(e)$ is an edge with endpoints $\phi(x)$ and $\phi(y)$, and
\item[(iii)] if $\phi(e)=e'$ then $\ell'(e')/\ell(e)$ is an integer.
  \end{itemize}
Given $e$ and $e'$ as in property (iii), we let $\mu_\phi(e)=\ell'(e')/\ell(e)$.
We say a morphism $\phi:G\rightarrow G'$ is \emph{harmonic} if for all $x\in V(G)$, we have that
$$m_\phi(x)=\sum_{\substack{e\in E(G), \\x\in e,\phi(e)=e'}}\mu_\phi(e)$$
is the same for all choices of $e'\in E(G')$ that are incident to $\phi(x)$.  Note that $m_\phi(x)\geq 0$; we say that $\phi$ is \emph{nondegenerate} if $m_\phi(x)>0$ for all $x\in V(G)$.  Finally, the \emph{degree} of a morphism $\phi$ is
$$\text{deg}(\phi)=\sum_{\substack{e\in E(G),\\ \phi(e)=e'}}\mu_\phi(e)$$
for any $e'\in E(G)$. In turns out that this integer is independent of the choice of $e'$ \cite[Lemma 2.4]{BN}

\begin{figure}[hbt]
\centering
\includegraphics[scale=0.8]{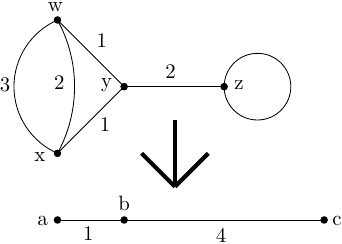}
\caption{A non-degenerate harmonic morphism of degree $2$}
\label{figure:morphism}
\end{figure}

As an example, consider the morphism $\phi$ illustrated in Figure \ref{figure:morphism}, which sends a graph of genus $3$ to a \emph{metric tree} (that is, a metric graph of genus $0$).  On vertices, we have $\phi(w)=\phi(x)=a$, $\phi(y)=b$, and $\phi(z)=c$.  The loop is collapsed to the point $c$, and the edges between $w$ and $x$ of lengths $3$ and $2$ are sent to the point $a$.  The edges connecting $y$ to $w$ and to $x$ are sent to the edge of length $1$, and the bridge of length $2$ is sent to the edge of length $4$.  This morphism is nondegenerate and harmonic, and has degree $2$.

We now define what it means for a graph to be hyperelliptic, following the conventions of \cite{Chan}.  Suppose that $G$ is a metric graph with no points of valence $1$ such that either: there exists a nondegenerate, harmonic morphism $\phi:G\rightarrow T$, where $T$ is a metric tree and $\deg(\phi)=2$; or $|V(G)|=2$.  Then we say that $G$ is a \emph{hyperelliptic graph}.  For instance, the genus $3$ metric graph in Figure \ref{figure:morphism} is hyperelliptic.  An alternate definition of hyperelliptic graphs builds up a theory of divisors on graphs, and calls a graph hyperelliptic if it has a rank $1$ divisor of degree $2$.  Yet another definition declares a metric graph $G$ is hyperelliptic if it has an involution $i$ such that $G/i$ is a tree.  The equivalence of these definitions is the content of \cite[Theorem 1.3]{Chan2}.

We remark that this is not the only notion of hyperelliptic graphs considered in the literature.  For instance, \cite{ABBR} considers those graphs that arise as tropicalizations of hyperelliptic curves; we call such a graph \emph{realizable}.  Indeed, they offer a complete classification of all realizable hyperelliptic graphs \cite[Corollary 4.5]{ABBR}.  This classification shows that the graph appearing in Figure \ref{hyp_step1} does not arise from a tropicalization, due to the three bridges meeting at a single vertex.  Whether or not we exclude such graphs will end up not effecting our end results, since all smooth tropical plane curves are realizable:  any tropical plane curve is the tropicalization of an algebraic curve, and if smooth faithfully represents the associated metric graph \cite[Corollary 5.28]{BPR}

\begin{figure}[hbt]
\centering
\includegraphics[scale=1]{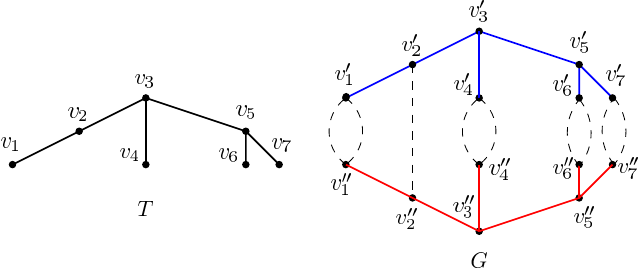}
\caption{A tree $T$ with $7$ vertices, and the corresponding ladder $G$ of genus 8, with the two copies of $T$ highlighted}
\label{figure:ladder}
\end{figure}

A key result we will use is the classification of all $2$-edge-connected hyperelliptic graphs from~\cite{Chan2}. Let $T$ be a metric tree with $g-1$ vertices, each of which has valence at most $3$.  Duplicate the tree, and connect the two copies by adding edges between corresponding vertices until the graph $G$ is trivalent, as illustrated in Figure \ref{figure:ladder}.  The resulting graph is called a \emph{ladder}.  There is a natural degree $2$ harmonic morphism $\phi$ from $G$ to $T$, defined as follows.  For every vertex $v\in T$ and the two corresponding vertices $v',v''\in G$, let $\phi(v')=\phi(v'')=v$.  Thus for each edge $vw\in T$, we must define $\phi(v'w')=\phi(v''w'')=vw$.  Finally, for every edge in $G$ of the form $v'v''$, let $\phi(v'v'')=v$.  For the graph $G$ in Figure \ref{figure:ladder}, this means $\phi(v_i')=\phi(v_i'')=v_i$ for all $i$, that the solid edges are sent to edges of $T$, and that the dotted edges are sent to vertices.

\begin{proposition}[Theorem 4.9 in \cite{Chan2}]  The $2$-edge-connected trivalent hyperelliptic graphs of genus $g$ are precisely the ladders of genus $g$.
\label{ladder_proposition}
\end{proposition}

\begin{figure}[hbt]
\centering
\includegraphics[scale=0.6]{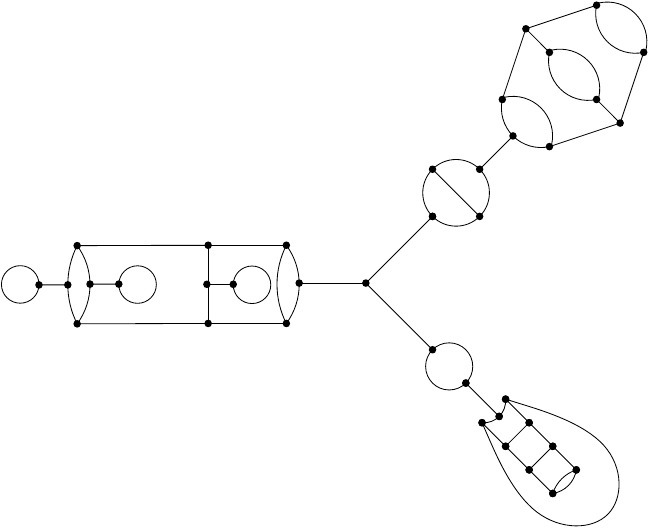}
\caption{A hyperelliptic graph that is not $2$-edge-connected}
\label{hyp_step1}
\end{figure}

This result allows us to construct \emph{all} trivalent hyperelliptic graphs. Deleting the bridges  of a hyperelliptic graph and smoothing over the $2$-valent vertices yields the $2$-edge-connected components, which must be ladders, genus $1$ loops, or simply points. Then, the bridges can only be attached at fixed points of the involution $i$, as discussed in Lemma 3.10 and Corollary 3.12 of \cite{Chan2}.  For a ladder constructed from a tree $T$, these fixed points are precisely the midpoints of the edges inserted between the two copies of $T$.  In the case that  a $2$-edge-connected component is a genus $1$ loop, there can be only two fixed points, splitting the loop into two  line segments of the same length.  An example of a hyperelliptic graph we can construct in this way is illustrated in Figure~\ref{hyp_step1}.

One special class of hyperelliptic graphs is the collection of \emph{chains}, which can be constructed as follows for genus $g\geq 2$.   Start with a line segment with $g-1$ vertices, where the $g-2$ edges have 
arbitrary positive lengths. Duplicate each edge so that
the resulting parallel edges have the same length, and attach two loops of arbitrary lengths
at the endpoints. (In the case of $g=2$, attach two loops to the single vertex.)
At this point, the graph has genus $g$ and contains $g-1$ vertices, all of which are $4$-valent.
There are two possible ways to split each vertex into two vertices connected
by an edge of arbitrary length, resulting in a trivalent graph called a \emph{chain of genus $g$}.  Ignoring lengths,  there are $2^{g-1}$ possible ways to perform this procedure; however, some give isomorphic graphs.  In particular, the number of combinatorial types of chains is equal to the number of binary strings of length $g-1$, with strings and their reverses identified.  Counting up such strings, we find that there are $2^{g-2}+2^{\lfloor(g-2)/2\rfloor}$ combinatorial types of chains of genus $g$.  The six combinatorial types of chains of genus $4$ are illustrated in Figure \ref{figure:g4chains}.  For each genus $g\geq 2$ there is a  unique $2$-edge connected chain of genus $g$, of them form illustrated in Figure \ref{figure:standardchain}.

\begin{figure}
\centering
\includegraphics[scale=0.7]{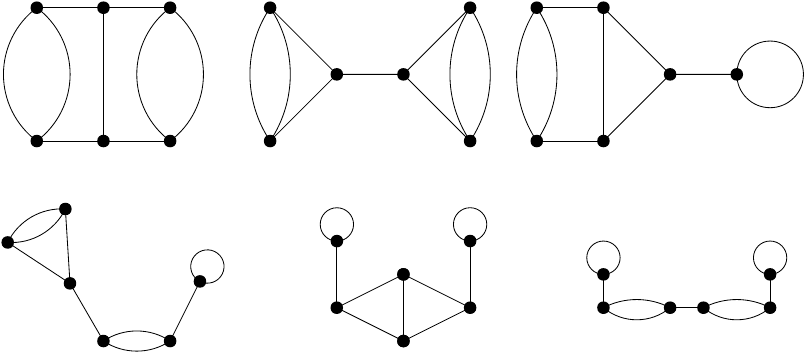}
\caption{The six chains of genus $4$}
\label{figure:g4chains}
 \end{figure}

We write $\mathbb{M}^{\textrm{chain}}_g$ for the
subset of $\mathbb{M}_g$ consisting of all chains. By construction, there are $2g-1$ degrees
of freedom for the edge lengths in a chain, so we have $\text{dim}(\mathbb{M}^{\textrm{chain}}_g)=2g-1$.
 In fact, we have $\mathbb{M}^{\textrm{chain}}_g\subset\mathbb{M}_{g,\textrm{hyp}}$: each chain has an involution sending the parallel edges to one another, and modding out by this involution yields a line segment.  Moreover, by \cite[Lemma 4.2]{BLMPR},  a graph with the same combinatorial type as a chain is hyperelliptic if and only if the two edges in each parallel pair have the same lengths.

We now turn to the background on $\mathbb{M}_g^{\textrm{planar}}$, the moduli space of tropical plane curves of genus $g$ introduced in \cite{BJMS}.  It is natural to decompose $\mathbb{M}_g^{\textrm{planar}}$ into smaller polyhedral spaces.  Let  $P$ be a lattice polygon with $g$ interior lattice points, and let $\Delta$ be a regular unimodular triangulation of $P$.  We will write $\mathbb{M}_\Delta$ for the cone of all metric graphs arising from $\Delta$, and $\mathbb{M}_P$ for all metric graphs arising from $P$.  Then we may write
$$\mathbb{M}_P=\bigcup_{\Delta}\mathbb{M}_\Delta,$$
and
$$\mathbb{M}_g^{\textrm{planar}}=\bigcup_{P}\mathbb{M}_P,$$
where the first union is taken over all regular unimodular triangulations $\Delta$ of $P$, and the second union is taken over all lattice polygons $P$ with $g$ interior lattice points.
It is worth noting that there are only finitely many lattice polygons $P$ with $g$ interior lattice points up to isomorphism \cite[Proposition 2.3]{BJMS}, so all unions can be taken to be finite.  Moreover, for computing $\mathbb{M}_g^{\textrm{planar}}$, it suffices to take the union over polygons $P$ that are \emph{maximal}, meaning that $P$ is not contained in any larger lattice polygon with the same configuration of interior lattice points.

As discussed in \cite[\S 1]{BJMS}, the space  $\mathbb{M}_g^{\textrm{planar}}$ has the structure of a stacky fan.  Their argument rests on considering a space $\mathbb{M}_{P,G}$, which is all metric graphs arising from $P$ with combinatorial type $G$, thought of as a subset of $\mathbb{R}_{\geq 0}^{3g-3}$.  This has the structure of a polyhedral fan; by choosing appropriate subdivisions, one can choose a fan structure that is invariant under the symmetries of $G$, allowing us to consider $\mathbb{M}_{P,G}$ as a stacky fan inside of $\mathbb{M}_g$.  Gluing together different $\mathbb{M}_{P,G}$'s yields $\mathbb{M}_g^{\textrm{planar}}$.  It is worth remarking that it is not clear if the stacky fan structure on $\mathbb{M}_g^{\textrm{planar}}$ is in any way canonical or unique, as there are many possible choices to make in the construction of it, e.g. in whether to include possible redundant polygons $P$. 

\begin{figure}
\centering
\includegraphics[scale=0.8]{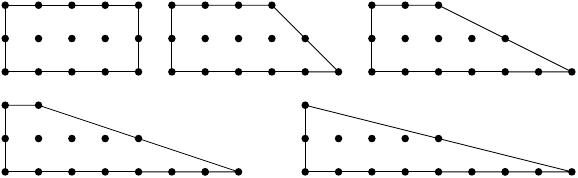}
\caption{The five maximal hyperelliptic polygons of genus $3$}
\label{figure:g3hyperelliptic_polygons}
\end{figure}

We define
$$ \mathbb{M}^{\textrm{planar}}_{g,{\textrm{hyp}}}  \,\, \, := \,\,\,
\bigcup_{P\textrm{ hyp}} \mathbb{M}_P, $$
where the union is over all hyperelliptic polygons $P$ of genus $g$.  Although by \cite{Ca} there are {$\frac{1}{6}(g+3)(2g^2+15g+16)$} hyperelliptic polygons of genus $g$, we may restrict our union to the maximal ones, of which there are $g+2$.  These are illustrated for $g=3$ in Figure \ref{figure:g3hyperelliptic_polygons}.  In fact, by \cite[Theorem 6.1]{BJMS}, the space $ \mathbb{M}^{\textrm{planar}}_{g,{\textrm{hyp}}}$ is equal to $ \mathbb{M}_{T^g_{\textrm{hyp}}}$, where $T^g_{\textrm{hyp}}$ is the maximal hyperelliptic triangle of genus $g$, illustrated for $g=3$ in Figure \ref{figure:hyperelliptic_triangle_example}.

\begin{figure}[hbt]
\centering
\includegraphics[scale=.8]{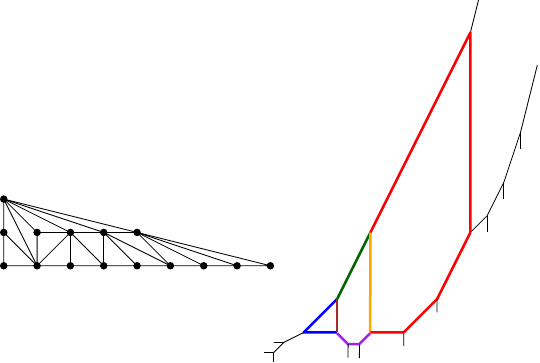}
\caption{The genus 3 hyperelliptic triangle $T^3_{\textrm{hyp}}$, with an example of a triangulation and a dual tropical curve}
\label{figure:hyperelliptic_triangle_example}
 \end{figure}

\begin{lemma}\label{lemma:chains} Let $P$ be a hyperelliptic polygon with genus $g$ at least $2$.  Then the skeleton of any smooth tropical curve with Newton polygon $P$ is a chain.  
\end{lemma}

\begin{proof}  It suffices to prove our claim where $P$ is one of the $g+2$ maximal hyperelliptic polygons, chosen to have $y$ coordinate between $0$ and $2$ and $x$ coordinate greater than or equal to $0$.  Suppose that in a triangulation of $P$, a line segment connects an interior point to the upper or lower boundary of $P$. Then by the structure of $P$, this line segment is either vertical or has a slope of the form $\pm 1/k$ for some integer $k$. It follows the dual edge of the tropical curve will have integer slope. Thus, the lattice length of this dual edge is equal to its horizontal width measured by the change in its $x$-coordinate.

Now, consider any regular unimodular triangulation of $P$. Because $P_\textrm{int}$ is a line segment, the triangulation will either connect two adjacent interior lattice points or separate them with an edge from the top boundary to the bottom boundary.  In a tropical curve dual to this triangulation, this results in a sequence of $g$ loops (one for each interior point), with successive loops either sharing edges (if the interior points were connected) or joined by a bridge (if the interior points were separated).  This means that the skeleton of the tropical curve will have the same combinatorial type as a chain.

To see that the skeleton must in fact be a chain, we need to show that parallel edges in the skeleton have the same length.  These parallel edges are made up of sequences of edges in the tropical curve, which start and stop at points with the same $x$-coordinates.  As discussed above, it follows that the sum of the lattice lengths of the strings of edges must be the same.  Thus the skeleton is in fact a chain.
\end{proof}

In the language of moduli spaces, this means that
$$\mathbb{M}^{\textrm{planar}}_{g,{\textrm{hyp}}} \subset \mathbb{M}^{\textrm{chain}}_g. $$

\section{Sprawling graphs and crowded graphs}
\label{section:sprawling_and_crowded}

In this section we describe two combinatorial obstructions to a graph being the skeleton of a smooth tropical plane curve.  In both cases, the obstructions are features of a graph that could not possibly arise in the dual graph to a triangulation of a lattice polygon.

\begin{definition}  A connected trivalent  graph $G$ is called \emph{sprawling}
 if there exists a vertex $s$ of $G$ such that $G \backslash \{s\}$ consists of three distinct components.
\end{definition}

  Note that each component of  $G\backslash \{s\}$ must have genus at least one: otherwise $G$ would not have been leafless, and hence not trivalent.   It is relatively simple and efficient to check whether a trivalent graph is sprawling:  the set of all bridges can be found in linear time \cite{Tar}, and then the graph is sprawling if and only if three of the bridges meet at a common vertex.   The sprawling graphs of genus at most $4$ are illustrated in Figure \ref{figure:sprawling_graphs}.

\begin{figure}[hbt]
\centering{
\includegraphics{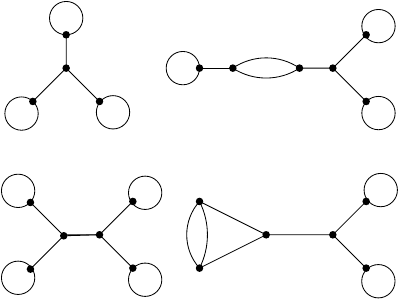}
}
    \caption{The four sprawling graphs of genus at most $4$}
  \label{figure:sprawling_graphs}
\end{figure}

\begin{lemma}  \label{prop:sprawling}
Sprawling graphs are never the skeletons of smooth tropical plane curves.
\end{lemma}
This was originally proven in \cite[Proposition 4.1]{CDMY}.  A proof also appears in \cite[Proposition~8.3]{BJMS}.

 The next obstruction is a priori more difficult to check, as it is a property that must hold over all planar embeddings of a graph.
 
\begin{definition} A planar embedding of a connected, trivalent planar graph $G$ is called \emph{crowded}
 if either:  there exist two bounded faces sharing at least two edges; or, there exists a bounded face sharing an edge with itself.  If all planar embeddings of such a $G$ are crowded, we say that $G$ is \emph{crowded}.
\end{definition}

Equivalently, we can say that a planar embedding is crowded if the dual planar graph is not simple; that is, if it has at least one loop or pair of parallel edges.The idea of crowdedness is related to the notion of strong regularity of CW complexes \cite[\S 2.2]{BSS}.  A CW complex is said to be \emph{strongly regular} if no face is glued to itself; if no two faces are glued together more than once; and if edges have distinct end-points, and pairs of edges have at most one endpoint in common.   An embedding of a graph is crowded if and only if one of the first two conditions is violated.  However, a noncrowded graph may have loops or bi-edges, which violate strong regularity.
 
\begin{example}\label{crowded_example} {Consider the graph $G$ in Figure \ref{figure:crowded_example}, shown with two different embeddings.  Combinatorially, these are the only two planar embeddings of $G$.  This can be seen by noting that for each of the three bi-edges in the graph, the two adjacent edges must point either both outwards or both inwards (otherwise the graph would have a bridge).  In fact, these edges can point inward for at most one bi-edge, and from there the embedding is determined.  Since both these embeddings are crowded, we conclude that $G$ is a crowded graph.
}
\end{example}

\begin{figure}[hbt]
\centering{
\includegraphics{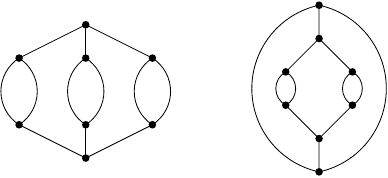}
}
    \caption{Two crowded embeddings of the same graph}
  \label{figure:crowded_example}
\end{figure}

 There are no crowded trivalent graphs of genus $g\leq 4$, as can be checked by consulting \cite{Bal}, which enumerates all trivalent connected graphs up to genus $6$.  In genus $5$ there are seven crowded graphs, depicted in Figure \ref{figure:genus5_nonrealizable_graphs}.  Proving crowdedness for these amounts to arguments to that of Example \ref{crowded_example}.

\begin{lemma}  \label{proposition:crowded}
Crowded graphs are never the skeletons of smooth tropical plane curves.
\end{lemma}

\begin{proof}  Suppose the skeleton of a smooth tropical plane curve is a planar embedding of a crowded graph $G$.  If the embedding has two bounded faces $F$ and $F'$ sharing at least two edges, then let $p$ and $p'$ be the corresponding interior lattice points of the tropical curve's Newton polygon.  Since $F$ and $F'$ share at least two edges, $p$ and $p'$ must be connected by at least two edges in the corresponding triangulation of the Newton polygon.  This is impossible, since the only possible edge between $p$ and $p'$ is the unique line segment connecting them.  A similar argument holds if the embedding of $G$ has a bounded face sharing an edge with itself.  These contradictions prove the claim.
\end{proof}

\begin{figure}[hbt]
\centering{
\includegraphics[scale=0.8]{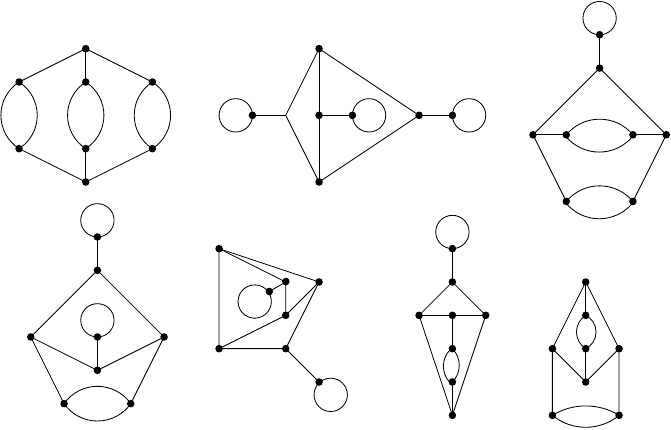}
}
    \caption{The seven crowded graphs of genus $5$}
  \label{figure:genus5_nonrealizable_graphs}
\end{figure}

We now develop several criteria for showing that a graph is crowded.

\begin{proposition}  Suppose $G$ is a planar graph obtained by connecting two connected planar graphs $G_1$ and $G_2$, each of genus at least one, with a pair of edges.  Let $G_1'$ be obtained from $G$ by deleting $G_2$ and replacing it with a bi-edge.  If $G_1'$ is crowded, then so is $G$.
\label{prop_surgery}
\end{proposition}

\begin{proof}  Assume that  $G_1'$ is crowded, and consider any planar embedding of $G$.  We will show that this is a crowded embedding. If the embedding has either $G_1$ or $G_2$ enveloping the other, then it is crowded since the inner graph has genus at least one, as illustrated in the first two images of Figure \ref{figure:crowded_surgery}. The possibility that one connecting edge is enveloped and one is not is ruled out by the supposition that $G_1$ and $G_2$ are connected. Thus we may assume the configuration is as in the third image. 

\begin{figure}[hbt]
\centering{
\includegraphics{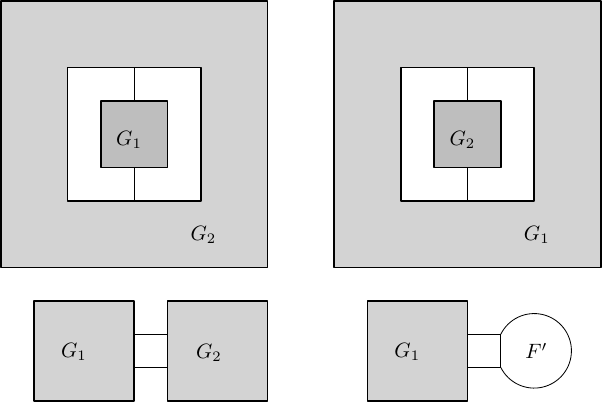}
}
    \caption{Possible configurations of $G_1$ and $G_2$ within $G$, the first two visibly crowded and the third yielding the resulting graph $G_1'$ on the bottom right}
  \label{figure:crowded_surgery}
\end{figure}

Delete $G_2$ from our embedding of $G$, and replace it with a bi-edge so that the bi-edge bounds a face $F'$ as illustrated, without wrapping either edge around $G_1$. Note that each bounded face of $G_1'$ besides $F'$  has a corresponding bounded face in $G$.  Since $G_1'$ is crowded, this embedding of $G_1'$ must be a crowded embedding, so either two bounded faces share two or more edges or some bounded face shares an edge with itself.  However, $F'$ cannot be any of these problematic faces, since it shares exactly one edge with exactly one bounded face, and none with itself. Thus the crowded configuration of faces must appear in $G$ as well, so this embedding of $G$ is crowded.  As this embedding was arbitrary, we conclude that $G$ is crowded.
\end{proof}

This result can be used to prove the following corollary, which will be key in determining which hyperelliptic graphs appear in tropical plane curves.

\begin{corollary}\label{crowded_corollary}  Let $G$ be a trivalent connected planar graph of the form illustrated in Figure \ref{figure:crowded_hyperelliptic}, where each unknown box contains a graph of genus at least one. Then $G$ is crowded.

\begin{figure}[hbt]
\centering
\includegraphics[scale=0.96]{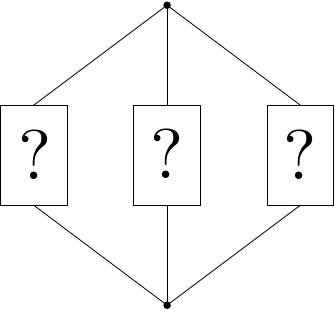}
\caption{A family of graphs that are always crowded}
\label{figure:crowded_hyperelliptic}
\end{figure}
\end{corollary}

\begin{proof}  Perform the surgery from Proposition \ref{prop_surgery} three times on $G$, thereby producing the graph from Example \ref{crowded_example}.  Since that graph is crowded, three applications of Proposition \ref{prop_surgery} imply that $G$ is crowded as well.
\end{proof}

\begin{example}\label{example:nd} Let $g\geq 5$. Using this corollary, we can construct an example of a metric graph $G$ that appears in $\textrm{trop}\left(\mathcal{M}_g^{nd}\right)$ but not in $\mathbb{M}_g^{\textrm{planar}}$, illustrating that  $\mathbb{M}_g^{\textrm{planar}}\subsetneq\textrm{trop}\left(\mathcal{M}_g^{nd}\right)$ for $g\geq 5$.  (This was shown to be true for $g=3$ and $g=4$ in \cite{BJMS}.)

Let $T$ be any metric tree with $g-1$ vertices such that the maximum degree of any vertex is equal to $3$; since $g-1\geq 4$, such a tree exists.  Construct the ladder $G$ associated to $T$.  This graph is hyperelliptic and bridgeless, and thus is the graph associated to a hyperelliptic algebraic curve by \cite[Corollary 4.15]{ABBR}.  As shown in \cite[\S 5]{CV}, any hyperelliptic curve over an algebraically closed field is nondegenerate, so $G\in \textrm{trop}\left(\mathcal{M}_g^{nd}\right)$.  However, $G$ has the structure from Corollary \ref{crowded_corollary} due to fact that $T$ has a degree $3$ vertex, implying that $G$ is crowded.  By Lemma \ref{proposition:crowded}, we have $G\notin \mathbb{M}_g^{\textrm{planar}}$.
\end{example}

Another way to show that a graph is crowded is based on particular subgraphs.  

\begin{lemma}\label{lemma:subgraph} Let $G'$ be a $2$-edge-connected component of a planar graph $G$.  If $G'$ is crowded, then so is $G$. 
\end{lemma}

\begin{proof} Choose any embedding of $G$, then delete everything that is not part of $G'$, smoothing over the resulting $2$-valent vertices.  Label the bounded faces $F_1,\ldots, F_k$.  Now add back in the rest of $G$.  Since $G'$ is a $2$-edge-connected component of $G$, the faces $F_1,\ldots, F_k$ are preserved as faces, and the number of edges shared by (not necessarily distinct)  pairs amongst $F_1,\ldots, F_k$ have either remained the same or increased.  Thus either two bounded faces of $G$ share two edges, or one bounded face shares an edge with itself, and the embedding is crowded.   Since this was an arbitrary embedding of $G$, we conclude that $G$ is crowded.
\end{proof}

 Note that is possible for a graph to have a crowded subgraph (perhaps with  $2$-valent vertices smoothed over)  without itself being crowded, as long as that subgraph is not a $2$-edge-connected component.  See Figure~\ref{figure:subgraph} for an example.
 
 \begin{figure}[hbt]
\centering
\includegraphics{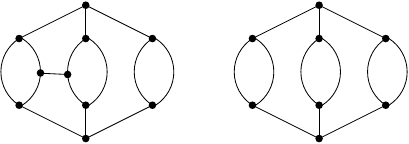}
\caption{A graph that isn't crowded, but has a crowded subgraph}
\label{figure:subgraph}
\end{figure}

\section{Combinatorial obstructions for hyperelliptic graphs}
\label{section:combinatorial}

Armed with our sprawling and crowded criteria, we are now ready to prove that chains are the only combinatorial types of hyperelliptic graphs that appear as the skeleton of a smooth tropical plane curve.  We begin with the case of $2$-edge-connected graphs.

\begin{proposition}\label{proposition:2-edge-connected} If a trivalent hyperelliptic graph is $2$-edge-connected, either it is a chain or it is crowded.
\end{proposition}

We remark that the $2$-edge-connected assumption is vital.  Without it, we could have any of the four graphs in Figure \ref{figure:sprawling_graphs}, all of which are hyperelliptic but none of which are chains or crowded.

\begin{proof} Let $G$ be a $2$-edge-connected hyperelliptic graph of genus $g$. By Proposition \ref{ladder_proposition}, $G$ is a ladder over a tree with $g-1$ vertices, each with valency at most three.  Note that $G$ is the $2$-edge-connected chain of genus $g$ if and only if the tree is a line segment.  Assume $G$ is not a chain.  Then the tree must contain a trivalent vertex and so $G$ is of the form shown in Figure~\ref{figure:crowded_hyperelliptic}, where each unknown box contains at least one bi-edge.  Corollary \ref{crowded_corollary} implies that $G$ is crowded.
\end{proof}

\begin{proposition}
\label{proposition:mustbechain}  If $G$ is the hyperelliptic skeleton of a smooth tropical plane curve, then $G$ is a chain.
\end{proposition}

Our proof of this proposition will show something even stronger:  that if a graph is the skeleton of a smooth tropical plane curve, and there exists \emph{some} metric on that graph that makes it hyperelliptic (not necessarily the metric given by that embedding), then that graph must be a chain under that new metric.  

\begin{proof}  Let $G$ be such a hyperelliptic skeleton, meaning it comes with a given embedding into the plane which cannot be a crowded embedding by Lemma \ref{proposition:crowded}.  Each 2-edge-connected component of $G$ must be hyperelliptic and  noncrowded by Lemma \ref{lemma:subgraph}, and so by Proposition \ref{proposition:2-edge-connected} each 2-edge-connected component of $G$ must be either a chain, a loop, or a vertex.

Let $G'$ be a  $2$-edge-connected component of $G$ of genus at least $2$. Then $G'$ inherits a non-crowded embedding into the plane from $G$.  Since $G'$ is a chain, it must be in the standard chain embedding illustrated in Figure~\ref{figure:standardchain}: any other embedding is crowded, as can be checked by inductively building the embedding loop by loop.  The only bridges that could possibly connect $G'$ to the rest of $G$ are a bridge from the middle of $e_0$ and a bridge from the middle of $e_g$.  This is because a $2$-edge-connected component connecting to $G$ from any other $e_i$, or with multiple edges from $e_0$ or from $e_g$, would make a bounded face of $G$ share an edge with itself, giving a crowded embedding.  Similarly, if $G'$ has genus $1$, it must have at most two adjacent edges, positioned in the exterior of the loop.

\begin{figure}[hbt]
\centering
\includegraphics{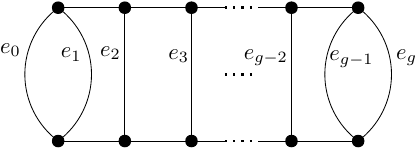}
\caption{The standard embedding of a chain, with vertical edges labelled $e_0$ to $e_g$}
\label{figure:standardchain}
\end{figure}

It follows that each $2$-edge-connected component of $G$ has at most one incoming and one outgoing edge.  As with any graph, shrinking the $2$-edge-connected components down to vertices yields a tree $T$.  Lemma \ref{prop:sprawling} implies that  $T$ must be a line segment: if $T$ had any trivalent vertices, the corresponding $2$-edge-connected component would have to be a vertex, and so  $G$ would be sprawling.  Considering the structure of each nontrivial $2$-edge-connected component, we conclude that $G$ must be a chain.
\end{proof}

\section{Metric obstructions}
\label{section:metric}

Proposition \ref{proposition:mustbechain} implies that $\mathbb{M}_g^{\textrm{planar}}\cap\mathbb{M}_{g,\textrm{hyp}}\subset \mathbb{M}_g^{\textrm{planar}}\cap \mathbb{M}_g^{\textrm{chain}}$.  However, it is not immediately clear that there is no contribution to $\mathbb{M}_g^{\textrm{planar}}\cap \mathbb{M}_g^{\textrm{chain}}$ from nonhyperelliptic polygons, which can give rise to graphs with the same combinatorial types as chains, and could \emph{a priori} have hyperelliptic metrics.  The following proposition rules this out, and is the last ingredient we need in order to prove Theorem \ref{thm:onlyhyperelliptic}.

\begin{proposition}  Let $G$ be the skeleton of a smooth tropical plane curve $C$ with nonhyperelliptic Newton polygon $P$.  If $G$ is combinatorially a chain, then the metric on $G$ is not hyperelliptic.
\label{proposition:chainmeanshyperelliptic}
\end{proposition}
In the proof of this proposition, when we say ``$G$ is a chain'', we mean that $G$ is a chain combinatorially, possibly with a nonhyperelliptic metric.

\begin{proof}  
Let  $\Delta$ be the unimodular triangulation of $P$ dual to the smooth tropical plane curve $C$ with skeleton $G$.
 The order on the distinguished cycles $c_1,\ldots,c_g$ of $G$ induces a natural ordering on the interior lattice points of $P$, which we will call $p_1,\ldots, p_g$.  Since $P$ is nonhyperelliptic, there exists some triple $(p_{i},p_{i+1},p_{i+2})$ of these interior lattice points that are not collinear.  We will assume that the cycle $c_{i+1}$ shares an edge with $c_{i}$ and an edge with $c_{i+2}$; the other cases with at least one bridge coming from $c_{i+1}$ are handled similarly.  Dually, this means that $\Delta$ contains the line segments $\overline{p_ip_{i+1}}$ and $\overline{p_{i+1}p_{i+2}}$.
 
Consider the triangle $T=\text{conv}(p_{i},p_{i+1},{p_{i+2}})$, which does not intersect the boundary of $P$ since its vertices are interior to $P$.  We claim that $T$ has area $\frac{1}{2}$.  If not, then by Pick's theorem, $T$ must contain at least one more interior lattice point $p'$, which must lie either in the interior of the triangle or on the edge $\overline{p_{i}p_{i+2}}$.  Either way, $\Delta$ must include an edge connecting $p_{i+1}$ to $p'$, violating the chain structure of $G$.  Thus, $T$ has area $\frac{1}{2}$, and so after a change of coordinates we may assume $p_i=(1,2)$, $p_{i+1}=(1,1)$, and $p_{i+2}=(2,1)$.  

Since $G$ is a chain, the triangulation  $\Delta$ does not contain the line segment $\overline{p_{i}p_{i+2}}$.  This means that some line segment in $\Delta$ containing $p_{i+1}$ must separate $p_{i}$ and $p_{i+2}$.  By the convexity of $P$, this means that the point $q=(2,2)$ is contained in $P$, and in fact $\overline{p_{i+1}q}$ is a line segment in $\Delta$.  Since $G$ is a chain, it follows that $q$ is a boundary point of $P$.  Since $P$ is convex, there is no segment in $\Delta$ containing $p_{i+1}$ that separates  $p_i$ from $q$, or $p_{i+2}$ from $q$. It follows that $\overline{p_iq}$ and $\overline{p_{i+2}q}$ are both segments in the triangulation $\Delta$.  In the dual tropical curve, let $e_h$ be dual to $\overline{p_{i}p_{i+1}}$; $e_v$ be dual to $\overline{p_{i+1}p_{i+2}}$; $e_1$ be dual to $\overline{p_{i+1}q}$; and $e_2$ be the remainder of the cycle $c_{i+1}$.  This is illustrated in Figure~\ref{figure:nonhyp}.  Let $\ell_h$, $\ell_v$, $\ell_1$, and $\ell_2$ denote the lengths of these edges, respectively.

\begin{figure}[hbt]
\centering
\includegraphics{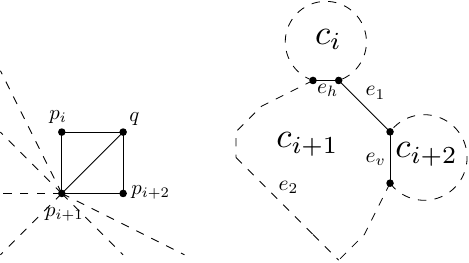}
\caption[A partial triangulation of a nonhyperelliptic polygon]{A portion of the triangulation $\Delta$ of $P$, and part of the dual tropical curve}
\label{figure:nonhyp}
\end{figure}

Let $q_1,q_2,\ldots,q_n$ denote the lattice points of $P$ that $\Delta$ connects to $p_{i+1}$, ordered counterclockwise starting with $q$ (so that $q_1=q$, $q_2=p_{i}$ and $q_n=p_{i+2}$).  Let $q_j=(a_j,b_j)$.  For $3\leq j\leq n-1$, at least one of $a_j$ and $b_j$ must be nonpositive due to the placement of $p_i$ and $p_{i+2}$.   In fact, we claim that for all $j$, at least one of $a_j$ and $b_j$ is equal to $0$. Suppose not.  Then some $q_j=(a_j,b_j)$ is in the interior of either the second, third, or fourth quadrant of $\mathbb{R}^2$.  Assume for the moment that it is the second quadrant.  Then the triangle $\text{conv}(p_i,p_{i+1},q_j)$ has area strictly greater than $1/2$ due to its base and height, meaning by Pick's theorem that the triangle must contain at least a fourth lattice point $p'$.  This point $p'$ must appear either in the interior of the triangle or on the edge $\overline{p_{i+1}q}$, meaning it must be an interior lattice point of $P$.   As there are no possible edges to separate them,  $p'$ is connected to $p_{i+1}$ by a segment in $\Delta$, which is impossible as $G$ is a chain, a contradiction. An identical argument holds for $q_j$ in the fourth quadrant, replacing $p_{i}$ with $p_{i+2}$.  Finally, if $q_j$ is in the third quadrant, we can reach a similar contradiction considering the triangle $\text{conv}(p_i,p_{i+2},q_j)$, which will have area strictly greater than $\frac{3}{2}$ and thus will contain an extra lattice point besides its vertices and the interior point $p_{i+1}$.  Thus, for all $j$, at least one of $a_j$ and $b_j$ is equal to $0$.

In the tropical embedding of the graph $G$, the edge $e_2$ is made up of line segments that are dual to $q_3,q_4,\ldots, q_{n-1}$.  Consider the line segments in $e_2$ dual to $q_i$'s of the form $(a_i,0)$.   The sum of the horizontal widths of these segments must be at least the sum of the horizontal widths of $e_1$ and $e_h$: otherwise the cycle $c_{i+1}$ would not be closed.  Since these line segments in $e_2$ have slopes in $\mathbb{Z}$, each of them has lattice length equal to horizontal width.  The same holds for $e_1$ and $e_h$, implying $\ell_2\geq \ell_1+\ell_h>\ell_1$.  This means $G$ has edges of a two-cut with different lengths, namely $e_1$ and $e_2$ with lengths $\ell_1\neq \ell_2$.  By \cite[Lemma 4.2]{BLMPR}, the graph $G$ cannot be hyperelliptic.
\end{proof}

It is worth remarking that it is not immediately obvious from Figure~\ref{figure:nonhyp} that  $e_2$ is longer than $e_1$, since we are considering lattice length rather than Euclidean length.  For instance, if $\ell_1=\ell_v=1$, $\ell_h=2$, and $e_2$ consists of a single line segment with slope $2/3$, then $\ell_1=\ell_2$.  This is ruled out by constraints on the lattice polygon $P$, but the result does require more work than it might initially seem.

The results of this section now allow us to prove that hyperelliptic graphs that arise as the skeletons of smooth tropical plane curves only come from hyperelliptic Newton polygons.

\begin{proof}[Proof of Theorem \ref{thm:onlyhyperelliptic}]  Let $C$ be a smooth tropical plane curve with Newton polygon $P$ and a hyperelliptic skeleton $G$.  By Proposition \ref{proposition:mustbechain}, the graph $G$ must be a chain.   If $P$ were not a hyperelliptic polygon, then by  Proposition \ref{proposition:chainmeanshyperelliptic} the chain $G$ could not be hyperelliptic as assumed.  We conclude that $P$ must be a hyperelliptic polygon.
\end{proof}

\medskip

\noindent\textbf{Acknowledgements.}  The author thanks several anonymous referees for very helpful comments on an earlier version of this paper.  
The author thanks Sarah Brodsky, Melody Chan, Michael Joswig, and Bernd Sturmfels for many helpful and illuminating conversations on tropical curves.  The author also thanks Desmond Coles, Neelav Dutta, Sifan Jiang, and Andrew Scharf for their help in developing the theory of crowded graphs.


\begin{thebibliography}{}
\bibitem{ACP}
Abramovich, D., Caporaso, L., Payne, S.:
{The tropicalization of the moduli space of curves}. Ann. Sci. \'{E}c. Norm. Sup\'{e}r {48} no. 4, 765--809 (2015)

\bibitem{ABBR}
Amini, O, Baker, M., Brugall\'{e}, E., Rabinoff, J.:
 Lifting harmonic morphisms II: Tropical curves and metrized complexes. Algebra Number Theory 9, no. 2, 267--315  (2015)

\bibitem{baker}
Baker, M.:  Specialization of linear systems from curves to graphs. With an appendix by Brian Conrad. Algebra Number Theory 2, no. 6, 613--653 (2008)

\bibitem{BLMPR}
Baker, M., Len, Y., Morrison, R., Pflueger, N., Ren, Q.:
{Bitangents of tropical plane quartic curves}. Math. Z. {282} no. 3-4, 1017--1031 (2016)

\bibitem{BN}
Baker, M., Norine, S.:
Riemann-Roch and Abel-Jacobi theory on a finite graph. Adv. Math. 215, no. 2, 766--788 (2007)

\bibitem{BPR}
Baker, M., Payne, S., Rabinoff, J.: Non-Archimedean geometry, tropicalization, and
metrics on curves. Algebr. Geom. 3, no. 1, 63-10 (2016)

\bibitem{Bal}
Balaban, A.T.: {Enumeration of cyclic graphs}. In: Balabad, A.T. (ed.) Chemical Applications of Graph Theory, pp. 63-105, Academic Press (1976)

\bibitem{BSS}
Bobenko, A.I., Sechelmann, S., Springborn, B.: Discrete Conformal Maps: Boundary Value Problems, Circle Domains, Fuschsian and Schottky Uniformization.  In: Bobenko, A.I. (ed.) Advances in discrete differential geometry, Springer, Berlin (2016)

\bibitem{BMV}
Brannetti, S., Melo, M., Viviani, F.: {On the tropical Torelli map}.  Advances in Mathematics
{226} no. 3, 2546-2586 (2011)

\bibitem{BJMS}
Brodsky, S., Joswig, M., Morrison, R.,  Sturmfels, B.:
{Moduli of tropical plane curves}.  Res. Math. Sci. 2, Art. 4, 31 pp (2015)

\bibitem{CDMY}
Cartwright, D., Dudzik, A., Manjunath, M.,  Yao, Y.:
{Embeddings and immersions of tropical curves}.
Collect. Math. {67} no. 1, 1--19 (2016)

\bibitem{Ca}
Castryck, W.:
{Moving out the edges of a lattice polygon}.
Discrete Comput. Geom.  47 no. 3, 496--518 (2012)

\bibitem{CC}
Castryck, W., Cools, F.: Newton polygons and curve gonalities. J. Algebraic Combin. 35, no. 3, 345--366 (2012)

\bibitem{CV}
Castryck, W.,  Voight, J.: On nondegeneracy of curves. Algebra Number Theory 3, no. 3, 255--281 (2009)
 
 \bibitem{Chan}
 Chan, M.: {Combinatorics of the tropical Torelli map}. Algebra and Number Theory {6},
1133-1169  (2012)

\bibitem{Chan_lectures}
Chan, M.: {Lectures on tropical curves and their moduli spaces. Moduli of curves}, 1--26, Lect. Notes Unione Mat. Ital., 21, Springer, Cham (2017)
 
\bibitem{Chan2} Chan, M.:
{Tropical hyperelliptic curves}.
J.~Algebraic Combin. {37}, 331--359 (2013) 

\bibitem{GK}
Gathmann, A., Kerber, M.: A Riemann-Roch theorem in tropical geometry. Math. Z. 259, no. 1, 217--230 (2008)


\bibitem{HMRT}
Hahn, M., Markwig, H., Ren, Y., Tyomkin, I.: Tropicalized quartics and canonical
embeddings for tropical curves of genus 3. Int. Math. Res. Not. IMRN (2019)


\bibitem{MS} Maclagan, D., Sturmfels, B.:
{Introduction to Tropical Geometry}.
Graduate Texts in Math., Vol.~161, American Math.~Soc. (2015)

\bibitem{MZ}
Mikhalkin, G., Zharkov, I.: Tropical curves, their Jacobians and theta functions. Curves and abelian varieties, 203--230, Contemp. Math., 465, Amer. Math. Soc., Providence, RI (2008)

\bibitem{NS}
Nishinou, T., Siebert, B.: Toric degenerations of toric varieties and tropical curves. Duke Math. J. 135, no. 1, 1--51 (2006)

\bibitem{dhruv}
Ranganathan, D.: Skeletons of stable maps I: rational curves in toric varieties. J. Lond. Math. Soc. (2) 95, no. 3, 804--832 (2017)

  \bibitem{Tar}  Tarjan, R.E.:
  {A note on finding the bridges of a graph}.
  Information Processing Lett., {2}, pp. 160--161 (1974)

\end{thebibliography}
\end{document}